\documentclass[12pt]{amsart}
\usepackage{amssymb}
\usepackage{amsmath}
\usepackage{amsfonts}
\usepackage{bbm}
\usepackage{hyperref}
\usepackage{cleveref}
\usepackage[dvipsnames]{xcolor}
\usepackage{todonotes}
\usepackage{mathtools}
\usepackage{enumitem}
\usepackage{fullpage}
\usepackage{ytableau}
\usepackage{pdfpages}
\usepackage{tikz}
\usetikzlibrary{arrows}
\usepackage{extarrows}
\usetikzlibrary{decorations.markings}
\usetikzlibrary{tikzmark,decorations.pathreplacing}
\usepackage{tikz-3dplot}
\usetikzlibrary{shapes.geometric,positioning,patterns}
\numberwithin{equation}{section}
\usepackage{float}
\makeatletter
\newcommand{\Vast}{\bBigg@{4}}
\newcommand{\Vvast}{\bBigg@{5}}
\newcommand{\vvast}{\bBigg@{3.5}}
\def\Ddots{\mathinner{\mkern1mu\raise\p@
\vbox{\kern7\p@\hbox{.}}\mkern2mu
\raise4\p@\hbox{.}\mkern2mu\raise7\p@\hbox{.}\mkern1mu}}
\makeatother

\DeclareMathOperator{\dm}{dim}

\DeclareMathOperator{\rk}{rk}

\DeclareMathOperator{\SSYT}{SSYT}

\newtheorem{thm}{Theorem}[section]

\newtheorem{cor}[thm]{Corollary}
\newtheorem{lem}[thm]{Lemma}
\newtheorem{defn}[thm]{Definition}
\newtheorem{rem}[thm]{Remark}

\newtheorem{eg}[thm]{Example}

\tikzstyle{bhor}=[rectangle, draw, thick, minimum width=0.25cm, minimum height=1cm]

\tikzstyle{br}=[rectangle, draw, thick, minimum width=3em, minimum height=3em]
\tikzstyle{ar}=[rectangle, draw, thick, minimum width=3.5em, minimum height=3.5em]

\title{A short bijective proof of dimension identities of Erickson and Hunziker}
\author{Nishu Kumari}
\address{Nishu Kumari, Department of Mathematics, 
Indian Institute of Science, Bangalore  560012, India.}
\email{nishukumari@iisc.ac.in}
\date{\today}

\begin{document}
\begin{abstract}
In a recent paper (arXiv:2301.09744), Erickson and Hunziker consider partitions in which the arm–leg difference is an arbitrary constant $m$. In previous works, these partitions are called $(-m)$-asymmetric partitions. Regarding
these partitions and their conjugates as highest weights, they prove an identity yielding an infinite family of dimension equalities between $\mathfrak{gl}_n$ and $\mathfrak{gl}_{n+m}$ modules. Their proof proceeds by the manipulations of the hook content formula. 
We give a simple bijective proof of their result.

% In a recent paper (arXiv:2301.09744, 2023), Erickson and Hunziker 
% prove an identity yielding an infinite family of dimension equalities between $\mathfrak{gl}_n$ and $\mathfrak{gl}_{n+m}$ modules, for fixed integers $n \geq 1$ and $m \geq 0$. 
% As a consequence, they characterize the partitions $\lambda$ such that the number of SSYT of shape $\lambda$ filled with entries in $\{1,2,\dots,n\}$ is the same as the number of SSYT of shape $\lambda'$ filled with entries in $\{1,2,\dots,n+m\}$.
% Their proofs proceed by manipulations of the hook content formula. 
% We give elegant bijective proof of their results. 
\end{abstract}


\maketitle

\section{Introduction}
Schur polynomials, denoted $s_{\lambda}(x_1,\dots,x_n)$ are  symmetric polynomials that play a central role in the representation theory of the symmetric groups, general linear groups, and unitary groups,
% algebraic combinatorics, 
exhibiting a wide range of elegant properties and finding numerous applications. In the literature, numerous identities can be found that contain Schur polynomials~\cite{sagan,macdonald-2015,littlewood-1950,James-Kerber-2009}.
% For instance, Littlewood-Richardson Rule, Cauchy Identity and its dual are some of them. 
% The representation theory of the symmetric groups, general linear groups, and unitary groups involves the occurrence of Schur polynomials. 
We recall the following classical identities due to Littlewood~\cite{littlewood-1950} that involve the Schur polynomials:
\begin{equation}
\label{litt-1}
\prod_{1 \leq i \leq j \leq n} (1-x_ix_j)= \sum_{\lambda} (-1)^{|\lambda|/2} s_{\lambda}(x_1,\dots,x_n),
\end{equation}
\begin{equation}
\label{litt-2}
 \prod_{1 \leq i < j \leq n+1} (1-x_ix_j)= \sum_{\lambda} (-1)^{|\lambda|/2} s_{\lambda'}(x_1,\dots,x_n,x_{n+1}),   
\end{equation}
where each sum is over partitions of length at most $n$ and of the form $(\alpha+1|\alpha)$, in Frobenius coordinates (see \cref{sec:pre} for precise definition).
% These identities can be viewed as Euler characteristics of Bernstein–Gelfand–Gelfand (BGG) complex for the trivial representation of classical Hermitian symmetric pairs. 
Recently Erickson and Hunziker~\cite{erick} have considered partitions 
of the form $(\alpha+m|\alpha)$, in Frobenius coordinates, known as $(-m)$-asymmetric partitions. 
By considering these partitions and their conjugates as the highest weights, they prove 
%an identity is established that results in 
%an infinite set of 
dimension equalities
% Regarding these partitions and their conjugates as highest weights, they prove an identity yielding an infinite family of dimension equalities
between $\mathfrak{gl}_n$ and $\mathfrak{gl}_{n+m}$ modules. Their proof proceeds by the manipulations of the hook content formula. 
As a consequence, they derive six
new families of identities (see~\cite[Table 4]{erick}) generalizing classical identities (\eqref{litt-1},\eqref{litt-2}).
In this article, we give a simple bijective proof of their dimension equalities. 
% A merit of bijective proofs is often that they reveal more about the relation between two
% types of objects than “just” the fact that they are counted by the same numbers.
% The plan for the rest of the paper is as follows.
We give all the definitions and preliminary results in \cref{sec:pre}. Next, we prove the dimension identities bijectively in \cref{sec:bij}.
\section{Preliminaries}
\label{sec:pre}
Recall that a {\em partition} $\lambda= (\lambda_1,\dots,\lambda_m)$ is a weakly decreasing sequence of nonnegative integers. The {\em weight} of a partition $\lambda$, denoted $|\lambda|$ is the sum of all of its parts and the {\em length} of a partition $\lambda$, denoted $\ell(\lambda)$ is the number of nonzero parts of $\lambda$. 
A partition $\lambda$ can be represented pictorially as a \emph{Young diagram}, whose $i^{\text{th}}$ row contains $\lambda_i$ left-justified boxes. 
For example, the Young diagram of $(4,2,2,1)$ is
\begin{equation}
\label{eg-ydiag}
\ydiagram{4,2,2,1}
\raisebox{-2em}{.}
\end{equation}

For a partition, $\lambda$, the {\em conjugate partition}, denoted $\lambda'$, is the partition whose Young diagram is obtained by transposing the Young diagram of $\lambda$.
The (Frobenius) {\em rank} of a partition $\lambda$, denoted $\rk(\lambda)$, is the largest integer $k$ such
that $\lambda_k \geq k$. The {\em Frobenius coordinates} of $\lambda$ are a pair of distinct partitions, denoted $(\alpha|\beta)$, of length $\rk(\lambda)$ given by $\alpha_i=\lambda_i-i$ and $\beta_j=\lambda'_j-j$. For example, the Frobenius
coordinates of the running example $(4,2,2,1)$ are $(3, 0|3, 1)$. 
% Let $\mathcal{P}$ be the set of partitions. 
For $a \in \mathbb{N}$ and a partition $\lambda=(\lambda_1,\dots,\lambda_m)$, by
% $\lambda \in \mathcal{P}$ 
$a+\lambda$ we will mean the partition $(a+\lambda_1,\dots,a+\lambda_m)$. 
\begin{defn}
\label{def:z-asym}
Let $z$ be an integer. We say that a partition $\lambda$ is \emph{$z$-asymmetric} if $\lambda=(\alpha|\, \alpha+z)$, in Frobenius coordinates for some strict partition $\alpha$. More precisely, 
$\lambda = (\alpha|\beta)$ where $\beta_i = \alpha_i + z$ for $1 \leq i \leq \rk(\lambda)$.
\end{defn}

The {\em $z$-asymmetric} terminology, as defined in~\cite{AYYER2022437}, is adopted in this work. 
Specifically, when $z=-1$, partitions that are $(-1)$-asymmetric are variously referred to as {\em almost self-conjugate}, 
{\em orthogonal}, or 
{\em shift-symmetric} in~\cite{erick}, \cite{AYYER2022437} and ~\cite{ALCO}, respectively.
Furthermore, a $(-1)$-asymmetric partition is the same as Macdonald’s double of a strict partition~\cite[p.14]{macdonald-2015} and Garvan–Kim–Stanton’s doubled partition of a strict partition $\alpha$, denoted $\alpha\alpha$~\cite[Sec. 8]{garvan-kim-stanton-1990}. 
In addition, the $1$-asymmetric partition is the same as the {\em threshold} partition in~\cite{Digjoy} and the {\em symplectic} partition in~\cite{AYYER2022437}. Note that the empty partition is $z$-asymmetric for all $z \in \mathbb{Z}$ vacuously. For example, the only $1$-asymmetric partitions of $6$ are $(3,1,1,1)$ and $(2,2,2)$.

For a cell $b = (i,j)$ in (the Young diagram of) $\lambda$, the \emph{hook length} is given by $h(b) = \lambda_i - i + \lambda'_j - j + 1$, which is the number of cells in its row to its right plus those in its column below it. 
The \emph{content} of $b$, denoted $c(b)$, is $j-i$. The \emph{arm} (resp. \emph{leg}) of $b$, denoted $a(b)$ (resp. $\ell(b)$), is the rightmost (resp. bottommost) cell in its row (resp. column). For example, the hook lengths and contents of the running example are\\
% \begin{equation}
% \label{eg-hook-content}
\[
\begin{ytableau}
7 & 5 & 2 & 1 \\
4 & 2 \\
3 & 1 \\
1
\end{ytableau}
\qquad 
\raisebox{-3.5ex}{and}
% \text{and}
\qquad
\begin{ytableau}
0 & 1 & 2 & 3 \\
-1 & 0 \\
-2 & -1 \\
-3
\end{ytableau}
\qquad
\raisebox{-3.5ex}{respectively.}
\]

% \end{equation}

Recall that a {\em semistandard tableau or tableau} of shape $\lambda$ with entries in $\{1,2,\dots,n\}$  is a filling of $\lambda$ such that the entries increase weekly along rows and strictly along columns. For a partition $\lambda$ of length at most $n$, we denote the set of semistandard tableaux of shape $\lambda$ filled with entries bounded by $n$ by $\SSYT_{\lambda}$.
% If entries increase strictly both along its columns and rows, then we call it a {\em standard tableau.}  
% The same hook content formula also gives 
% and is given by
% \begin{equation}
%  \label{hc}
%  \prod_{u \in \lambda} \frac{(n+c(u))}{h(u)}.   
% \end{equation}
Suppose $F^{(n)}_{\pi}$ is a finite-dimensional simple $\mathfrak{gl}_n$ module with highest weight $\pi$, where $\pi$ is a partition of length at most $n$. Then the dimension of $ F^{(n)}_{\pi}$, denoted $\dm F^{(n)}_{\pi}$ is equal to the 
cardinality of $\SSYT_{\pi}$,
% number of semistandard Young tableaux of fixed shape $\pi$ filled with entries bounded by $n$,
and is given by the following hook content formula~\cite{stanley1971theory}.
\begin{equation}
\label{dim}
 \dm F^{(n)}_{\pi}  = \prod_{u \in \pi} \frac{n+c(u)}{h(u)}.   
\end{equation}

% \begin{rem}
% a
% \end{rem}

% The character of an irreducible polynomial
% representation of $\text{GL}_n(\mathbb{C})$ corresponding to the partition $\lambda$ is the Schur function $s_{\lambda}(x_1,x_2,\dots,x_n)$. 
For a partition $\lambda=(\lambda_1,\ldots,\lambda_n)$, the \emph{Schur polynomial} $s_{\lambda}(x_1,x_2,\dots,x_n)$ 
% or \emph{general linear (type A) character} of $\GL_n(\mathbb{C})$ 
is given by
\begin{equation}
 \label{gldef}
s_\lambda(X)=\frac{\displaystyle\det_{1\le i,j\le n}\Bigl(x_i^{\lambda_j+n-j}\Bigr)}
{\displaystyle\det_{1\le i,j\le n}\Bigl(x_i^{n - j}\Bigr)},
\end{equation}
where the denominator is the standard Vandermonde determinant,
\begin{equation}
\label{gldenom}
\det_{1\le i,j\le n}\Bigl(x_i^{n - j}\Bigr)
= \prod_{1\le i<j\le n}(x_i-x_j).
\end{equation}

Stanley~\cite{stanley} specialized the hook content formula given in \eqref{dim} and showed that the principle specialisation of the Schur polynomial $s_{\lambda}(1,q,\dots,q^{n-1})$ can be expressed as a
product involving the hook lengths and contents of the boxes in the diagram for $\lambda$. 
% This gives a generating function for the semistandard tableaux of shape $\lambda$ with entries in $\{1,2,\dots,n\}$. 
Given a partition $\lambda$, define 
\begin{equation}
\label{k}
k(\lambda)=\sum_{i} (i-1) \lambda_i = \sum_{i} \binom{\lambda'_i}{2}.    
\end{equation}
By \cite[Theorem 7.21.2]{stanley}, we have
\begin{equation}
\label{spe-schur}
s_{\lambda}(1,q,\dots,q^{n-1}) = q^{k(\lambda)} \prod_{u \in \lambda} \frac{[n+c(u)]}{[h(u)]}, 
% s_{\lambda'}(\underbrace{1,\dots,1}_{n+m})
\end{equation}
where $[a]=1+q+\dots+q^{a-1}$, $a \in \mathbb{N}$.
 
\begin{lem}
\label{lem:k}
Let $\lambda$ be an $m$-asymmetric partition.
% of length at most $n$
Then 
\[
    2(k(\lambda)-k(\lambda'))= m|\lambda|.
    \]
\end{lem}
\begin{proof} First, note that $\lambda=(\alpha|\alpha+m)$, for a strict partition $\alpha$ if and only if $\lambda$ is of the form
\begin{equation*}
    \lambda  % =(\alpha|\alpha+m)
    =(\alpha_1+1,\dots,\alpha_r+r,\underbrace{r,\dots,r}_{\alpha_r+m },
 \underbrace{r-1,\dots,r-1}_{\alpha_{r-1}-\alpha_r-1 },\dots,
 \underbrace{1,\dots,1}_{\alpha_1-\alpha_2-1}).
\end{equation*}

In that case, its conjugate is
\begin{equation*}
    \lambda'
    = (\alpha_1+m+1,\dots,\alpha_r+m+r,\underbrace{r,\dots,r}_{\alpha_r },
 \underbrace{r-1,\dots,r-1}_{\alpha_{r-1}-\alpha_r-1 },\dots,
 \underbrace{1,\dots,1}_{\alpha_1-\alpha_2-1}).
\end{equation*}
%   \[
% 2(k(\lambda)-k(\lambda')) = 2 \sum_{i=1}^{m+n} (i-1) (\lambda_i-\lambda_i')
% \]
% and 
% \[
% k(\lambda)-k(\lambda') = \sum_{i=1}^{m+n} (i-1) (\lambda_i-\lambda_i')= \sum_{i=1}^{r} m(i-1) 
% +\sum_{i=r+1}^{r+\alpha_r} (i-1) r - \sum_{i=r+\alpha_r+1}^{r+\alpha_{r-1}-1} (i-1) (r-1)
% + \sum_{i=r+\alpha_{r-1}}^{r+\alpha_{r-2}-2} (i-1) (r-2)
% + \dots + \sum_{i=\alpha_2+3}^{\alpha_1+1} (i-1) 
% \]
% \[

% - \sum_{i=r+1}^{r+\alpha_r+m} (i-1) r 
% - \sum_{i=r+\alpha_r+m+1}^{r+\alpha_{r-1}+m-1} (i-1) (r-1)
% - \sum_{i=r+\alpha_{r-1}+m}^{r+\alpha_{r-2}+m-2} (i-1) (r-2)
% \]
% \[
% - \dots + \sum_{i=\alpha_2+m+3}^{\alpha_1+m+1} (i-1)  
% \]

By \eqref{k}, we have 
\begin{equation}
    \label{12}
    \begin{split}
    k(\lambda)-k(\lambda') 
    = \sum_{i=1}^{m+n} (i-1) (\lambda_i-\lambda_i') & \\
 = - \sum_{i=1}^{r} m(i-1) + & \sum_{i=r+\alpha_r+1}^{r+\alpha_r+m} (i-1) r  
+  \sum_{i=1}^{r-1} im(\alpha_{i}-\alpha_{i+1}-1).
\end{split}
\end{equation}
% \begin{equation}
% \label{12}
%     \begin{split}
%   k(\lambda)-k(\lambda') 
%     = \sum_{i=1}^{m+n} (i-1) (\lambda_i-\lambda_i') 
% & = \sum_{i=1}^{r} m(i-1) -  \sum_{i=r+\alpha_r+1}^{r+\alpha_r+m} (i-1) r \\  
% - & m(r-1)(\alpha_{r-1}-\alpha_r-1)-m(r-2)(\alpha_{r-2}-\alpha_{r-1}-1)-\dots-m(\alpha_{1}-\alpha_2-1).      
%     \end{split}
% \end{equation}

We note that 
\begin{equation}
\label{13}
-\sum_{i=1}^{r} m(i-1)  + \sum_{i=r+\alpha_r+1}^{r+\alpha_r+m} (i-1) r = \frac{mr(r+m)}{2} + mr \alpha_r.     
\end{equation}
Using \eqref{13} in \eqref{12}, we have
\begin{equation}
\label{15}
k(\lambda)-k(\lambda')  =   \frac{mr(r+m)}{2} + mr \alpha_r  + \sum_{i=1}^{r-1} im(\alpha_{i}-\alpha_{i+1}-1).
\end{equation}
Now observe that 
\begin{equation}
\label{14}
|\lambda|= r^2+mr+2r\alpha_r+ \sum_{i=1}^{r-1} 2i(\alpha_{i}-\alpha_{i+1}-1).    
\end{equation}
% 2(r-1)(\alpha_{r-1}-\alpha_r-1)+\dots+2(\alpha_{1}-\alpha_2-1).
So, simplifying \eqref{15} using \eqref{14}, we get
\begin{equation*}
     k(\lambda)-k(\lambda')  = \frac{mr(r+m)}{2} + mr \alpha_r  + \frac{m}{2} \left(|\lambda| - r^2 - mr - 2r\alpha_r \right) 
 = \frac{m}{2} |\lambda|,
\end{equation*}
% \begin{equation*}
%     \begin{split}
%      k(\lambda)-k(\lambda') 
%      &=
%   - \frac{m}{2} \left(|\lambda| - r^2 - mr - 2r\alpha_r \right) \\
%  &= \frac{-mr(r+m)}{2} - mr \alpha_r  - \frac{m}{2} \left(|\lambda| - r^2 - mr - 2r\alpha_r \right) 
%  = - \frac{m}{2} |\lambda|,    
%     \end{split}
% \end{equation*}
completing the proof.
% \[
% \frac{\prod_{u \in \lambda} q^{c(u)-1}}
% {\prod_{u \in \lambda'} q^{c(u)-1}} 
% = \frac{\prod_{i \in Z} q^{ix_i(\lambda)}}
% {\prod_{u \in \lambda'} q^{}} 
% = q^{-m|\lambda|}.
% \] 
% % \newpage
% Since $\sum_{i=1}^r \lambda_i=r^2+r \alpha_r+(r-1)(\alpha_{r-1}-\alpha_r-1)+\dots+(\alpha_{1}-\alpha_2-1)$, 
% \[
% = \sum_{i=1}^{r} m(i-1) - \sum_{i=r+\alpha_r+1}^{r+\alpha_r+m} (i-1) r +mr^2+mr \alpha_r - m\sum_{i=1}^r \lambda_i
% \]
% \[
% \frac{mr(r-1)}{2}-r\left(m(r+\alpha_r)+\frac{m(m-1)}{2}\right)+mr^2+mr \alpha_r - m\sum_{i=1}^r \lambda_i
% \]
% \[
% \frac{mr(r-m)}{2}-m\sum_{i=1}^r \lambda_i
% \]
% \[
% 2(k(\lambda)-k(\lambda')) = m\left(r(r-m) - 2\sum_{i=1}^r \lambda_i\right)
% \]
% Using $\sum_{i=1}^r \lambda_i=r^2+r \alpha_r+(r-1)(\alpha_{r-1}-\alpha_r-1)+\dots+(\alpha_{1}-\alpha_2-1)$
% \[
% = m\left(-mr - \sum_{i=1}^r \lambda_i - r \alpha_r-(r-1)(\alpha_{r-1}-\alpha_r-1)-\dots-(\alpha_{1}-\alpha_2-1) \right)
% \]
% \[
% -m|\lambda|.
% \]
\end{proof}

For a partition $\lambda$, by~\cite[Chapter I.1, Example 3]{macdonald-2015}, we have 
\[
\sum_{u \in \lambda} c(u)=k(\lambda')-k(\lambda). 
\]
So, the following corollary holds.
\begin{cor}
\label{cor:content1}
Let $\lambda$ be an $m$-asymmetric partition.
% of length at most $n$
Then 
    \[\displaystyle \sum_{u \in \lambda} c(u) - \displaystyle \sum_{u \in \lambda'} c(u) = -m|\lambda|\]
\end{cor}

To give a bijective proof of Stanley's hook content formula, Krattenthaler~\cite{KRATTENTHALER199966} introduced the notion of a content tabloid and a hook tabloid of a fixed shape. We note down the definitions below. 
% of the content tabloid and the hook tabloid. 

\begin{defn}
Let $\lambda$ be a partition of length at most $n$. A {\em content tabloid} of shape $\lambda$ is a filling of the cells of $\lambda$ with
integers such that the entry in each cell $b$ of $\lambda$ is at least $1-c(b)$ and at most $n$. We denote the set of content tabloids of shape $\lambda$ by $C_{\lambda}$. Also, the sum of the entries of some content tabloid $T$ is called the norm of $T$ and denoted by $n(T)$.
% , for any
% cell $b$ of $\lambda$.
\end{defn}

The number of content tabloids of shape $\lambda$ is: 
\begin{equation}
\label{count:content}
   \prod_{u \in \lambda} (n+c(u)). 
\end{equation}

\begin{defn}
Let $\lambda$ be a partition. A {\em hook tabloid} of shape $\lambda$ is a filling of the cells of $\lambda$ with
integers such that the entry in each cell $b$ of $\lambda$ is at least $-a(b)$ and at most $\ell(b)$.
% , for any
% cell $b$ of $\lambda$.
We denote the set of hook tabloids of shape $\lambda$ by $H_{\lambda}$.
\end{defn}

The number of hook tabloids of shape $\lambda$ is: 
\begin{equation}
\label{count:hook}
  \prod_{u \in \lambda} (a(u)+\ell(u)+1)=
\prod_{u \in \lambda} h(u).  
\end{equation}

\begin{eg} Suppose $\lambda=(4,2,2,1)$ is a partition of length at most $4$. Then the number of content tabloids and hook tabloids of the shape $\lambda$ are 60480 and 1680 respectively. The following figure illustrates a content tabloid and hook tabloid of shape $\lambda$:
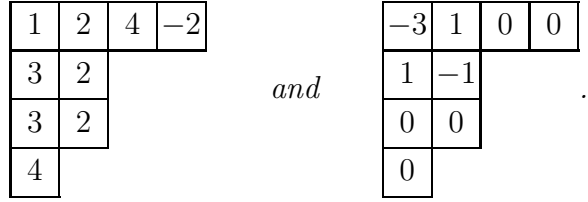
\[
\begin{ytableau}
1 & 2 & 4 & -2\\
3 & 2 \\
3 & 2\\
4 
\end{ytableau} 
\qquad 
\raisebox{-3.5ex}{and}
\qquad
\begin{ytableau}
-3 & 1 & 0 & 0\\
1 & -1 \\
0 & 0\\
0 
\end{ytableau}
\raisebox{-3.5ex}{.}
% \qquad
% \raisebox{-3.5ex}{respectively.}
\]
\end{eg}
\begin{defn}
    We say that a bi-infinite sequence $X=(\dots,x_{-2},x_{-1},x_0,x_1,x_2,\dots)$ is a {\em content sequence} if it satisfies the following conditions:
    \begin{itemize}
    \item the sequence is a unimodal sequence with the maximum at the origin, i.e. $\dots \leq x_{-2} \leq x_{-1} \leq x_0 \geq x_1 \geq x_2 \geq \dots$, 
    \item the sequence contains each number from the set $\{1,\dots,x_0-1\}$ at least once on both sides of $x_0$. 
\end{itemize}
\end{defn}
We call $x_0$ the peak of the content sequence $(\dots,x_{-2},x_{-1},x_0,x_1,x_2,\dots)$.
We underline the origin to avoid any ambiguity. For example $(\dots,0,1,2,2,2,\underline{3},3,2,1,0,\dots)$ is a content sequence with the peak $3$. The set of all content sequences is denoted as $\mathcal{C}$.
\begin{lem}
  There exists a one-to-one correspondence between the set $\mathcal{P}$ of all partitions and the set $\mathcal{C}$ of all content sequences.
\end{lem}
% \begin{lem}
%     Then $x(\lambda) = (x_i(\lambda))_{i=-\infty}^{\infty}$ is a content sequence.
% \end{lem}
% \begin{defn}
% Let $x_i(\lambda)$, for $i \in \mathbb{Z}$ be the number of boxes in $\lambda$ with content $i$. The {\em content sequence} of the partition $\lambda$ is a bi-infinite sequence 
% \[
% x(\lambda)=(\dots,x_{-2}(\lambda),x_{-1}(\lambda),x_0(\lambda),x_1(\lambda),x_2(\lambda),\dots).
% \]
% \end{defn}
\begin{proof}
Define a map $x:\mathcal{P} \rightarrow \mathcal{C}$ by 
% Let $\lambda=(\alpha|\beta)$ be a partition with rank $r$. Suppose 
% \[
% x(\lambda) \coloneqq (x_i(\lambda))_{i=-\infty}^{\infty},\]
\[
(x(\lambda))_i \coloneqq \text{number of boxes in $\lambda$ with content $i$, $i \in \mathbb{Z}$.}
\]
% $x_i(\lambda)$ is the number of boxes in $\lambda$ with content $i$, $i \in \mathbb{Z}$. 
Then $x(\lambda)$ will be of the form    
\begin{equation}
\label{x-1}
\begin{split}
x(\lambda) 
= (\dots,
\underbrace{1,\dots,1}_{\beta_1-\beta_2}, \dots, 
\underbrace{r-1,\dots,r-1}_{\beta_{r-1}-\beta_r},
&\underbrace{r,\dots,r}_{\beta_r}, 
\underline{r}, \\
&\underbrace{r,\dots,r}_{\alpha_r},
 \underbrace{r-1,\dots,r-1}_{\alpha_{r-1}-\alpha_r},
 \dots,
 \underbrace{1,\dots,1}_{\alpha_1-\alpha_2},
 \dots),  
 \end{split}
\end{equation}
where $\lambda=(\alpha|\beta)$ be a partition with rank $r$ (see \cref{eg-content}).  
% $\psi(\lambda) \coloneqq x(\lambda)$. 
Since both $\alpha$ and $\beta$ are strict partitions, $x(\lambda)$ is a content sequence.
% with the peak $r$.
% at the origin $x_0(\lambda)$.
So, $x$ is well defined. Also, it is easy to see that for a content sequence $c \in \mathcal{C}$, there exists a unique partition $\lambda$ such that $x(\lambda)=c$.
Hence, $x$ is invertible.
% there is a unique partition 
\end{proof}
\begin{rem}
\label{cotent} 
We call $x(\lambda)$ the {\em content sequence of $\lambda$}.
 A partition $\lambda$ is of the form $(\alpha+m|\alpha)$ for a strict partition $\alpha$ and an integer $m \geq 0$ $($i.e $\lambda'$ is $m-$asymmetric$)$ if and only if its content sequence is
\begin{equation*}
(\dots,\underbrace{1,\dots,1}_{\alpha_1-\alpha_2}, \dots, \underbrace{r-1,\dots,r-1}_{\alpha_{r-1}-\alpha_r},   
 \underbrace{r,\dots,r}_{\alpha_r},
 \underline{r},\underbrace{r,\dots,r}_{\alpha_r+m},
 \underbrace{r-1,\dots,r-1}_{\alpha_{r-1}-\alpha_r},\dots,
 \underbrace{1,\dots,1}_{\alpha_1-\alpha_2},\dots).    
\end{equation*}
\end{rem}
\begin{eg}
\label{eg-content}
Suppose $\lambda=(3,0|3,1)$. The following figure illustrates the Young diagram of $\lambda$ where each box is filled with its content.
\[
\begin{ytableau}
0 & 1 & 2 & 3 \\
-1 & 0 \\
-2 & -1 \\
-3
\end{ytableau}
\raisebox{-1.5em}{.}
\]
The content sequence of $\lambda$ is $({1,1},{2},\underline{2},{1,1,1}).$
% The following figure illustrates 
\end{eg}
\section{Bijective Proofs of Dimension Identities}
\label{sec:bij}
% Let $X_i(\lambda)$ be the  
For the formulation of the following bijection, it is convenient to fix a 
labelling of the cells of the Young diagram of a partition $\lambda$. For $a \in \mathbb{Z}$, we label the cells along a diagonal (i.e the cells with the same content) with the content $a$ as $(i,a)$, $1 \leq i \leq x_a(\lambda)$, where $x_a(\lambda)$ is the number of boxes in $\lambda$ with the content $a$. 
% , and cells with the content $i$ come before cells with the content $j$, for all $i<j$. 
The following figure 
illustrates this labelling of the cells along each diagonal for
$\lambda = (5,3,1)$. 
\begin{figure}[h]
\begin{center}
$$
\begin{matrix}
\resizebox{!}{3.4cm}{%
\begin{tikzpicture}[node distance=0 cm,outer sep = 0pt]
\node[br] (1) at (0,0) {\Large {\color{white}-}(1,0)};
 \node[br] (3) at (1.5, 0) {\Large (1,1)}; 
  \node[br] (6) at (2.9,0) {\Large (1,2)};
  \node[br] (5) at (4.3, 0) {\Large (1,3)};
  \node[br] (4) at (5.7,0) {\Large (1,4)};
  \node[br] (11) at (0,-1.25) {\Large (1,-1)};
 \node[br] (13) at (1.5, -1.25) {\Large (2,0)}; 
  \node[br] (16) at (2.9,-1.25) {\Large (2,2)};
  \node[br] (21) at (0,-2.5) {\Large (1,-2)};
  % \draw [thick] (-1,-0.5)--(-1,-2.5)--(-2,-2.5)--(-2,-3.5);
  % \draw [thick] (-2,-0.5)--(-2,-1.5)--(-3,-1.5)--(-3,-3.5);
\end{tikzpicture} 
}
\end{matrix}
$$
\end{center}
\end{figure}
% % \[
% \begin{figure}[htbp!]
%     \centering
% {
%  % \ytableausetup{mathmode,boxsize=3em} 
% \begin{ytableau}
% (1,0) & (1,1) & (1,2) & (1,3) & (1,4)\\
% (1,-1) & (2,0) & (2,1)\\
% (1,-2) 
% \end{ytableau}}
% \end{figure}
% \]
\begin{lem} 
\label{lem:bij}
Let $\lambda=(\alpha|\beta)$ be a partition of length at most $n$. 
% Suppose $n(T)$ is the sum of entries of a content tabloid $T \in C_{\lambda}$.
% and
% $C_{\lambda}$ be the set of content tabloids of shape $\lambda$. 
Then there exists a bijection $\phi:C_{(\alpha+m|\beta)} \rightarrow C_{(\alpha|\beta+m)}$ such that $n(\phi(T))=n(T)+m|\lambda|$. 
% is the same as that of shape $\lambda'$ (filled with entries at most $m+n$) iff $\lambda=(\alpha+m|\alpha)$.
\end{lem}
\begin{proof}
Suppose $\lambda^{(m)} \coloneqq (\alpha+m|\beta)$ and $\lambda_{(m)} \coloneqq (\alpha|\beta+m)$ are two partitions of length at most $n$ and $n+m$ respectively. 
Define a map $\phi:C_{\lambda^{(m)}} \rightarrow C_{\lambda_{(m)}}$ by 
\begin{equation}
\label{phi}
\phi(T)(i,a) = T(i,a-m)+m, \, \quad a \in \mathbb{Z}, \, 1 \leq i \leq x_{a}(\lambda^{(m)}),
\end{equation}
where $x_{a}(\lambda^{(m)})$ is the number of boxes in $\lambda^{(m)}$ of content $a$.
% where $c(b')=c(b)-m$, and $b'$ and $b$ have the same order in the total order defined above. 
If $
% x(\lambda)=
(\dots,x_{-2},x_{-1},x_0,x_1,x_2,\dots)$ is the content sequence of $\lambda$, then observe that the content sequences of $\lambda^{(m)}$ and $\lambda_{(m)}$ are 
\[
(\dots,x_{-2},x_{-1},
\underline{x_0},\underbrace{x_0,\dots,x_0}_{m},x_1,x_2,\dots)
\] 
and 
\[(\dots,x_{-2},x_{-1},
\underbrace{x_0,\dots,x_0}_{m}, \underline{x_0},x_1,x_2,\dots)
\]
respectively.
Since the content sequence of a partition $\lambda^{(m)}$ differs from that of $\lambda_{(m)}$ by a shift of length $m$ (i.e \, $x_{a}(\lambda^{(m)}) = x_{a-m}(\lambda_{(m)}), \, a \in \mathbb{Z}$) and $a \leq \phi(T)(i,a) \leq m+n$, the map $\phi$ is well defined. Also, by \eqref{phi}, one can see that $\phi$ is clearly invertible and $n(\phi(T))=n(T)+m|\lambda|$ holds.
% For converse, same as in the previous paper. if possible assume $\lambda = (\beta|\alpha) \neq (\alpha+m|\alpha)$ and $i \in [1,\rk(\lambda)]$ such that $\beta_j=\alpha_j+m$ for $j<i$ and $\beta_i \neq \alpha_i+m$.\\
% Case 1. $\beta_i > \alpha_i+m$.
\end{proof}
\begin{eg} 
\label{eg-bijection}
Suppose $\lambda=(3,0|2,0)$, $n=3$ and $m=1$. Let $\lambda^{(1)}=(4,1|2,0)$ and $\lambda_{(1)}=(3,0|3,1)$ be two partitions of length at most $3$ and $4$ respectively. The content sequences of $\lambda^{(1)}$ and $\lambda_{(1)}$ are $(1,1,\underline{2},2,1,1,1)$
and $(1,1,2,\underline{2},1,1,1)$ respectively. The following figure illustrates the image of a content tabloid $T$ of shape $\lambda^{(1)}$ under the bijection $\phi$ given in \cref{lem:bij}.
\begin{figure}[htbp!]
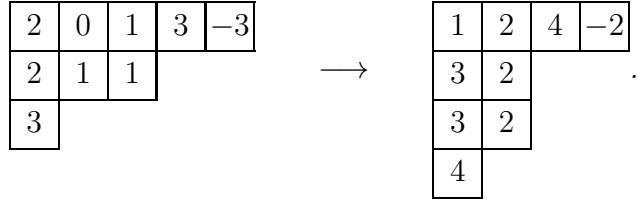

\begin{center}
\[
 % \ytableausetup{mathmode, boxframe=normal, boxsize=1.5em} 
 \begin{ytableau}
2 & 0 & 1 & 3 & -3  \\
2 & 1 & 1  \\
3  
\end{ytableau}
\qquad
\raisebox{-2.0ex}{$\longrightarrow$} 
\qquad
\begin{ytableau}
1 & 2 & 4 & -2\\
3 & 2 \\
3 & 2\\
4 
\end{ytableau}
\raisebox{-2ex}{.}
\]
% \label{fig}
\end{center}
\caption{Content tabloids $T$ and $\phi(T)$}
\end{figure}
\end{eg}
\begin{rem}
In~\cite{KRATTENTHALER199966}, Krattenthaler gave an involution principle-free bijective proof of the hook content formula \eqref{dim}. For a partition $\pi$ of length at most $n$, he described a bijection $\gamma_{\pi} : 
H_{\pi} \times \SSYT_{\pi}  \rightarrow C_{\pi}$. 
Suppose $\phi:C_{(\alpha+m|\beta)} \rightarrow C_{(\alpha|\beta+m)}$ is the bijection defined in \cref{lem:bij}. 
Then 
\[ \gamma_{(\alpha|\beta+m)}^{-1} \circ \phi \circ \gamma_{(\alpha+m|\beta)} : 
H_{(\alpha+m|\beta)} \times \SSYT_{(\alpha+m|\beta)} 
\rightarrow
H_{(\alpha|\beta+m)} \times \SSYT_{(\alpha|\beta+m)}
\]
will also be a bijection. However, the following example shows that for a fixed hook tabloid $h_0 \in H_{(\alpha+m|\beta)}$,
the map $\psi:\SSYT_{(\alpha+m|\beta)} \rightarrow \SSYT_{(\alpha|\beta+m)}$
by 
\[
\psi(T) = p(\gamma_{(\alpha|\beta+m)}^{-1} \circ \phi \circ \gamma_{(\alpha+m|\beta)}(h_0,T)),
\]
where $p:H_{(\alpha|\beta+m)} \times \SSYT_{(\alpha|\beta+m)} \rightarrow \SSYT_{(\alpha|\beta+m)}$ is a projection map and is given by $p(h,S)=S$, is not a bijection.
% the restriction of $\gamma_{(\alpha+m|\beta)} \circ \phi \circ \gamma_{(\alpha|\beta+m)}^{-1}$ to $\{h_0\} \times \SSYT_{(\alpha+m|\beta)}$ will not give us a bijection between $\SSYT_{(\alpha+m|\beta)}$ and $\SSYT_{(\alpha|\beta+m)}$.
\end{rem}
\begin{eg} Suppose $(\alpha|\beta)=(1|1)$, $n=2$, $m=1$ and $\phi:C_{(2|1)} \rightarrow C_{(1|2)}$ is the bijection defined in \cref{lem:bij}. 
Let $\gamma_{(2|1)} : 
H_{(2|1)} \times \SSYT_{(2|1)}  \rightarrow C_{(2|1)}$ and $\gamma_{(1|2)} : 
H_{(1|2)} \times \SSYT_{(1|2)}  \rightarrow C_{(1|2)}$
be bijections defined in~\cite{KRATTENTHALER199966}. 
Fix $h_0$ to be the hook tabloid of shape $(2|1)$, where each box is filled with zero. Then we have  
\[
\Vast( \begin{ytableau}
0 & 0 & 0\\
        0
    \end{ytableau} \,\, , \,\, \begin{ytableau}
     1 & 1 & 2\\
        2   
    \end{ytableau} \Vast) \xlongrightarrow{\gamma_{(2|1)}}  
    \begin{ytableau}
        1 & 1 & 2\\
        2
    \end{ytableau} \xlongrightarrow{\phi}  
    \begin{ytableau}
        2 & 3 \\
        2\\
        3
    \end{ytableau} \xlongrightarrow{\gamma^{-1}_{(1|2)}} 
   \raisebox{0em}{ \Vvast( \raisebox{1em}{\text{\begin{ytableau}
        1 & 0 \\
        0\\
        0
    \end{ytableau} \raisebox{-0.75em}{,} \begin{ytableau}
        1 & 3 \\
        2\\
        3
    \end{ytableau}}} \Vvast)} \raisebox{0em}{,}
\]
and 
\[
\Vast( \begin{ytableau}
0 & 0 & 0\\
        0
    \end{ytableau} \,\, , \,\, \begin{ytableau}
     1 & 2 & 2\\
        2   
    \end{ytableau} \Vast) \xlongrightarrow{\gamma_{(2|1)}}  
    \begin{ytableau}
        1 & 2 & 2\\
        2
    \end{ytableau} \xlongrightarrow{\phi}  
    \begin{ytableau}
        3 & 3 \\
        2\\
        3
    \end{ytableau} \xlongrightarrow{\gamma^{-1}_{(1|2)}} 
   \raisebox{0em}{ \Vvast( \raisebox{1em}{\text{\begin{ytableau}
        1 & 0 \\
        1\\
        0
    \end{ytableau} \raisebox{-0.75em}{,} \begin{ytableau}
        1 & 3 \\
        2\\
        3
    \end{ytableau}}} \Vvast)}.
\]
So, $\psi=p(\gamma^{-1}_{(1|2)} \circ \phi \circ \gamma_{(2|1)}): \SSYT_{(2|1)} \rightarrow \SSYT_{(1|2)}$ is not injective:
\[
\psi \vvast(\begin{ytableau}
     1 & 2 & 2\\
        2   
    \end{ytableau} \vvast) = \psi \vvast(\begin{ytableau}
     1 & 2 & 2\\
        2   
    \end{ytableau} \vvast) = \raisebox{1em}{\text{\begin{ytableau}
        1 & 3 \\
        2\\
        3
    \end{ytableau}}}.
\]
\end{eg}
Since the number of content tabloids of a fixed shape is given by \eqref{count:content}, we have the following corollary of \cref{lem:bij}. 
\begin{cor} 
\label{cor:content}
Let $\lambda = (\alpha|\beta)$ be a partition of length at most $n$. Then
    \[
    \prod_{u \in (\alpha+m|\beta)} (n+c(u)) =
    \prod_{u \in (\alpha|\beta+m)} (m+n+c(u)). 
    \]
\end{cor}
\begin{rem}
    If $m \neq 0$, then the product of the hook lengths of all the cells in $(\alpha+m|\beta)$ may not be equal to that in $(\alpha|\beta+m)$. So, by \eqref{dim}, we note that $\dm F^{(n)}_{(\alpha+m|\beta)}$ may not be the same as $\dm F^{(m+n)}_{(\alpha|\beta+m)}$.
\end{rem}

From \cref{lem:bij}, we obtain the following equality between the generating functions of content tabloids of shape of a partition and its conjugate.
\begin{cor}
\label{cor:gf} Let $\lambda
% =(\alpha+m|\alpha)
$ be a $m-$asymmetric partition of length at most $m+n$. 
% Suppose $n(T)$ is the sum of entries of a content tabloid $T \in C_{\lambda}$. 
Then
\[
\sum_{T \in C_{\lambda}} q^{n(T)} 
=  
\sum_{T \in C_{\lambda'}} 
q^{n(T)+m|\lambda|}.
\]
\end{cor}
\begin{thm}[{\cite[Theorem 2.1]{erick}}] 
\label{thm:dim}
Let $(\alpha|\beta)$ be a partition of length at most $p$ and with the first part at most $q$. 
% = [p \times q]$. 
% $\ell(\lambda) \leq p$ and $\ell(\lambda') \leq q$. 
Then 
 \[
 \dm F^{(p)}_{(\alpha+m|\beta)} 
 \dm F^{(q)}_{(\beta+m|\alpha)} 
 =
 \dm F^{(p+m)}_{(\alpha|\beta+m)} 
 \dm F^{(q+m)}_{(\beta|\alpha+m)} 
\]
\end{thm}
\begin{proof} 
By \eqref{dim}, we have
\begin{equation}
\label{5}
  \dm F^{(p)}_{(\alpha+m|\beta)} 
 \dm F^{(q)}_{(\beta+m|\alpha)}  = 
 \prod_{u \in (\alpha+m|\beta)} \frac{p+c(u)}{h(u)} 
\prod_{(v) \in (\beta+m|\alpha)} \frac{q+c(v)}{h(v)}.  
\end{equation}

% \[
% \prod_{(i,j) \in (\alpha+m|\beta)} \frac{p+c(i,j)}{h(i,j)} 
% \prod_{(i,j) \in (\beta+m|\alpha)} \frac{q+c(i,j)}{h(i,j)}
% =
% \prod_{(i,j) \in (\alpha|\beta+m)} \frac{p+m+c(i,j)}{h(i,j)}
% \prod_{(i,j) \in (\beta|\alpha+m)} \frac{q+m+c(i,j)}{h(i,j)}
% \]
% \[
% \prod_{(i,j) \in (\alpha+m|\beta)} (p+c(i,j))
% \prod_{(i,j) \in (\beta+m|\alpha)} (q+c(i,j))
% =
% \prod_{(i,j) \in (\alpha|\beta+m)} (p+m+c(i,j))
% \prod_{(i,j) \in (\beta|\alpha+m)} (q+m+c(i,j))
% \]
By \cref{cor:content} and the fact that conjugate partitions have the same multiset of hook lengths, we see that the right side of \eqref{5} is the same as
\[
 \prod_{u \in (\alpha|\beta+m)} \frac{p+m+c(u)}{h(u)} 
\prod_{(v) \in (\beta|\alpha+m)} \frac{q+m+c(v)}{h(v)}.
\]
So, by \eqref{dim} \cref{thm:dim} holds true.
\end{proof}
% Bijection between $C_{\lambda^{(m)}} \times C_{\lambda'^{(m)}} \rightarrow C_{\lambda_{(m)}} \times C_{\lambda'_{(m)}}$
% \[
% (T,S) \rightarrow
% \]
\begin{thm}[{\cite[Theorem 2.2]{erick}}]  Let $\lambda$ be a partition of length at most $n$. Then 
   we have $\dim F^{(n)}_{\lambda}=\dim F^{(m+n)}_{\lambda'}$ for each $n \geq \ell(\lambda)$ if and only if $\lambda'$ is $m-$asymmetric. 
   % =(\alpha+m|\alpha)$ for some strict partition $\alpha$.
\end{thm}
\begin{proof}
    Suppose $\dim F^{(n)}_{\lambda}=\dim F^{(m+n)}_{\lambda'}$ for all $n \geq \ell(\lambda)$. Since conjugate partitions have the same multiset of hook lengths, the hook–content formula \eqref{dim} yields that the number of content tabloids of shape $\lambda$ and $\lambda'$ are the same. So, using \eqref{count:content}, we have 
    \begin{equation}
    \label{11}
      \prod_{u \in \lambda} (n+c(u)) 
    =
    \prod_{u \in \lambda'} (n+m+c(u)),  
    \end{equation}
  for all $n \geq \ell(\lambda)$. This is the same as 
    \begin{equation}
    \label{112}
      \prod_{i \in \mathbb{Z}} (n+i)^{x_i(\lambda)} 
    =
     \prod_{j \in \mathbb{Z}} (n+m+j)^{x_j(\lambda')} 
     = 
     \prod_{j \in \mathbb{Z}} (n+m-j)^{x_{j}(\lambda)}, 
    \end{equation}
 for all $n \geq \ell(\lambda)$. The last equality follows from the fact that $x_j(\lambda')=x_{-j}(\lambda)$. Since \eqref{112} is true for all $n \geq \ell(\lambda)$,
 % if 
 % $\lambda'$ has $x_j(\lambda')$ box of content $j$, then $\lambda$ has  box of content $-j$
    % \begin{equation}
    %     \prod_{i \in \mathbb{Z}} (n-i)^{x_i(\lambda)} 
    % =
    % \prod_{i \in \mathbb{Z}} (n+m+i)^{x_i(\lambda)}
    % \end{equation}
      % this implies
     % $x_i(\lambda)=r,$ $-m \leq i \leq 0$ and
     $x_i(\lambda)=x_{-i+m}(\lambda)$ for all $i$. Therefore by the unimodality of the content sequence and \cref{cotent}, $\lambda'$ is $m-$asymmetric.
     % $\lambda=(\alpha+m|\alpha)$. 
     To prove the converse, assume $\lambda
   =(\alpha+m|\alpha)$ for some strict partition $\alpha$.
     Then the converse easily follows from \cref{thm:dim} by taking $\beta=\alpha$ and $p=q=n$, completing the proof.
\end{proof}
We have the following $q$-analogue. It is stated in a slightly different language than remark 2.3 in~\cite{erick}. 
\begin{thm}
% [{\cite[Remark 2.3]{erick}}] 
\label{thm:3}
Let $\lambda$ be a $m-$asymmetric partition of length at most $m+n$. Then
\[
s_{\lambda}(q^{1-n-m},q^{3-n-m},\dots,q^{n+m-3},q^{n+m-1}) = s_{\lambda'}(q^{1-n},q^{3-n},\dots,q^{n-3},q^{n-1}).
\]
\end{thm}
% Given a partition $\lambda$, define 
% \[
% k(\lambda)=\sum_{i} (i-1) \lambda_i = \sum_{i} \binom{\lambda'_i}{2}.
% \]
\begin{proof} 
By \eqref{gldef}, we have
\begin{equation*}
    \begin{split}
       s_{\lambda}(q^x,q^{x+1},\dots,q^{x+n-1}) 
     &  =
\frac{
\det \left( (q^{x+i})^{\lambda_j+n-j} \right)}
{\det \left( (q^{x+i})^{n-j} \right)
} \\
& = \frac{q^{\sum_j x(\lambda_j+n-j)}  
\det \left( q^{i(\lambda_j+n-j)} \right)}
{q^{\sum_j x(n-j)}  
\det \left( q^{i(n-j)} \right)} 
\\
% =q^{\sum_j x(\lambda_j)} s_{\lambda}(1,q,\dots,q^{n-1})
& = q^{x |\lambda| } s_{\lambda}(1,q,\dots,q^{n-1}), \quad  \text{ for each } x \in \mathbb{R}.  
    \end{split}
\end{equation*}
So, we note that
\begin{align*}
   s_{\lambda}(q^{1-n-m},\dots,q^{n+m-1})  
% = s_{\lambda}((q^2)^{\frac{1-n}{2}},\dots,(q^2)^{\frac{n-1}{2}}) 
&= s_{\lambda}((q^2)^x,(q^2)^{x+1},\dots,(q^2)^{x+n+m-1}) \quad \\
\text{and} \quad
s_{\lambda'}(q^{1-n},\dots,q^{n-1})
&= s_{\lambda'}((q^2)^y,(q^2)^{y+1},\dots,(q^2)^{y+n-1}), 
\end{align*}
% \[
% s_{\lambda}(q^{1-n},\dots,q^{n-1}) 
% % = s_{\lambda}((q^2)^{\frac{1-n}{2}},\dots,(q^2)^{\frac{n-1}{2}}) 
% = s_{\lambda}((q^2)^x,(q^2)^{x+1},\dots,(q^2)^{x+n-1}) \quad \text{and} \quad
% s_{\lambda}(q^{1-n-m},\dots,q^{n+m-1}) 
% = s_{\lambda}((q^2)^x,(q^2)^{x+1},\dots,(q^2)^{x+n+m-1})
% \]
where $x=\frac{1-m-n}{2}$, $y=\frac{1-n}{2}$. Therefore, it is sufficient to show 
\begin{equation}
  q^{2x|\lambda|} s_{\lambda}(1,q^2,\dots,q^{2(m+n-1)})
=  q^{2y|\lambda'|} s_{\lambda'}(1,q^2,\dots,q^{2n-2}),  
\end{equation}
which in turn is equivalent to 
\begin{equation}
    s_{\lambda}(1,q^2,\dots,q^{2(m+n-1)})
% =  q^{2|\lambda|(y-x)} s_{\lambda}(1,q^2,\dots,q^{2(m+n-1)})
=
q^{m |\lambda|} s_{\lambda'}(1,q^2,\dots,q^{2n-2})  
\end{equation}
or 
\begin{equation}
\label{2}
    s_{\lambda}(1,q,\dots,q^{m+n-1})
% =  q^{2|\lambda|(y-x)} s_{\lambda}(1,q^2,\dots,q^{2(m+n-1)})
=
q^{\frac{m |\lambda|}{2}} s_{\lambda'}(1,q,\dots,q^{n-1})  
\end{equation}
By \eqref{spe-schur}, we have 
\[
s_{\lambda}(1,q,\dots,q^{m+n-1}) = q^{k(\lambda)} \prod_{u \in \lambda} \frac{[m+n+c(u)]}{[h(u)]} 
% s_{\lambda'}(\underbrace{1,\dots,1}_{n+m})
\]
and 
\[
s_{\lambda'}(1,q,\dots,q^{n-1}) = q^{k(\lambda')} \prod_{u \in \lambda'} \frac{[n+c(u)]}{[h(u)]}.
\]
% where $[a]=1+q+\dots+q^{a-1}$, $a \in \mathbb{Z}$. 
Now observe that
\begin{equation}
\label{1}
    \begin{split}
       \prod_{u \in \lambda} 
       [m+n+c(u)]
       = \prod_{u \in \lambda} (1+q+&\dots+q^{m+n+c(u)-1}) 
      \\
      = \left( \prod_{u \in \lambda} 
       q^{c(u)-1} \right) &
\sum_{T \in C_{\lambda}} q^{n(T)} 
 =  \left( \prod_{u \in \lambda} 
       q^{c(u)-1} \right) 
\sum_{T \in C_{\lambda'}} 
q^{n(T)+m|\lambda|} 
\\
&=
q^{m|\lambda|} 
\left(
\frac{\prod_{u \in \lambda} q^{c(u)-1}}
{\prod_{u \in \lambda'} q^{c(u)-1}}
\right)
\prod_{u \in \lambda'} [n+c(u)],
    \end{split}
\end{equation}
where $C_{\lambda}$ is the set of content tabloids of shape $\lambda$, $n(T)$ is the sum of entries of a content tabloid $T$ and the third equality follows from \cref{cor:gf}.
% So,
% \[
% s_{\lambda}(1,q,\dots,q^{n-1}) = q^{k(\lambda)-m|\lambda|-k(\lambda')} 
% \frac{\prod_{u \in \lambda} q^{c(u)-1}}
% {\prod_{u \in \lambda'} q^{c(u)-1}}
% s_{\lambda'}(1,q,\dots,q^{m+n-1})
% \]
Since $\lambda$ is $m$-asymmetric, by \cref{cor:content1} we note that
\[
\frac{\prod_{u \in \lambda} q^{c(u)-1}}
{\prod_{u \in \lambda'} q^{c(u)-1}} = q^{-m|\lambda|}.
\]
Substituting in \eqref{1}, we see that
\[
\prod_{u \in \lambda} 
       [m+n+c(u)]= \prod_{u \in \lambda'} 
       [n+c(u)].
\]
Hence,
\[
s_{\lambda}(1,q,\dots,q^{m+n-1}) = q^{k(\lambda)-k(\lambda')} 
s_{\lambda'}(1,q,\dots,q^{n-1}).
\]
So by \cref{lem:k}, \eqref{2} holds, completing the proof.
\end{proof}
\section*{Acknowledgements}

I would like to thank my PhD advisor A. Ayyer for all the insightful discussions.

% \bibliographystyle{alpha}
% \bibliography{Bibliography}

\bibliographystyle{alpha}
\bibliography{ref}

\begin{thebibliography}{GKS90}

\bibitem[AK22]{AYYER2022437}
Arvind Ayyer and Nishu Kumari.
\newblock Factorization of classical characters twisted by roots of unity.
\newblock {\em Journal of Algebra}, 609:437--483, 2022.

\bibitem[EH23]{erick}
William~Q. Erickson and Markus Hunziker.
\newblock Dimension identities, almost self-conjugate partitions, and {BGG}
  complexes for hermitian symmetric pairs, 2023.

\bibitem[GKS90]{garvan-kim-stanton-1990}
Frank Garvan, Dongsu Kim, and Dennis Stanton.
\newblock Cranks and {$t$}-cores.
\newblock {\em Invent. Math.}, 101(1):1--17, 1990.

\bibitem[JK09]{James-Kerber-2009}
Gordon~Douglas James and Adalbert Kerber.
\newblock {\em The Representation Theory of the Symmetric Group}, volume~16 of
  {\em Encyclopedia of Mathematics and its Applications}.
\newblock Cambridge University Press, 2009.

\bibitem[Kra99]{KRATTENTHALER199966}
C.~Krattenthaler.
\newblock Another involution principle-free bijective proof of {S}tanley's
  hook-content formula.
\newblock {\em Journal of Combinatorial Theory, Series A}, 88(1):66--92, 1999.

\bibitem[Lit06]{littlewood-1950}
Dudley~E. Littlewood.
\newblock {\em The theory of group characters and matrix representations of
  groups}.
\newblock AMS Chelsea Publishing, Providence, RI, 2006.
\newblock Reprint of the second (1950) edition.

\bibitem[LPS20]{ALCO}
Svante Linusson, Samu Potka, and Robin Sulzgruber.
\newblock On random shifted standard {Young} tableaux and 132-avoiding sorting
  networks.
\newblock {\em Algebraic Combinatorics}, 3(6):1231--1258, 2020.

\bibitem[Mac15]{macdonald-2015}
I.~G. Macdonald.
\newblock {\em Symmetric functions and {H}all polynomials}.
\newblock Oxford Classic Texts in the Physical Sciences. The Clarendon Press,
  Oxford University Press, New York, second edition, 2015.
\newblock With contribution by A. V. Zelevinsky and a foreword by Richard
  Stanley, Reprint of the 2008 paperback edition [MR1354144].

\bibitem[PPS23]{Digjoy}
Joseph Pappe, Digjoy Paul, and Anne Schilling.
\newblock The {B}urge correspondence and crystal graphs.
\newblock {\em European Journal of Combinatorics}, 108:103640, 2023.

\bibitem[Sag01]{sagan}
Bruce~E. Sagan.
\newblock {\em The Symmetric Group}.
\newblock Graduate Texts in Mathematics. Springer New York, NY, second edition,
  2001.

\bibitem[SF99]{stanley}
Richard~P. Stanley and Sergey Fomin.
\newblock {\em Enumerative Combinatorics}, volume~2 of {\em Cambridge Studies
  in Advanced Mathematics}.
\newblock Cambridge University Press, 1999.

\bibitem[Sta71]{stanley1971theory}
Richard~P Stanley.
\newblock Theory and application of plane partitions.
\newblock {\em Stud. Appl. Math}, 50(1971):167--188, 1971.

\end{thebibliography}

\end{document}